\documentclass[reqno, a4paper]{amsart}
\usepackage{amssymb, amscd, amsfonts,  amsthm, stmaryrd}
\usepackage[mathscr]{eucal}
\usepackage{graphicx}
\usepackage[all,poly]{xy} 
\usepackage[all]{xy} 

\newtheorem{theorem}{Theorem}[section]
\newtheorem{lemma}[theorem]{Lemma}
\newtheorem{corollary}[theorem]{Corollary}

\theoremstyle{definition}

\newtheorem{remark}[theorem]{Remark}

\newcommand{\luk}{\L u\-ka\-s\-ie\-wicz}

\newcommand{\remove}[1]{}

\DeclareMathOperator{\McN}{\mathcal M}
\DeclareMathOperator{\conv}{\rm conv}

\DeclareMathOperator{\I}{[0,1]}
\DeclareMathOperator{\cube}{[0,1]^{\it d}}

\DeclareMathOperator{\oneset}{\rm oneset}

\DeclareMathOperator{\support}{\rm support}
\DeclareMathOperator{\uplevelset}{\rm uplevelset}

\title[Euler characteristic and valuations on MV-algebras]
{The Euler valuation on MV-algebras}

\author[D. Mundici]{Daniele Mundici$^\dag$}
\address[D. Mundici]{Department of Mathematics and Computer
Science  ``Ulisse Dini'' \\
University of Florence \\
viale Morgagni 67/A \\
50134 Florence \\
Italy}

\email{mundici@math.unifi.it }

\author[A.Pedrini]{Andrea Pedrini $^\ddag$}
\address[A.Pedrini]{Department of Mathematics
 ÒF. EnriquesÓ\\ University of Milan
 \\ via Cesare Saldini 50\\ 20133 Milan, 
Italy }
\email{andrea.pedrini@unimi.it}



  \date{\today}

\begin{document}

\begin{abstract}
Every finitely presented MV-algebra $A$  has a unique
idempotent valuation  $\mathsf{E}$ assigning value 1 to every basic
element of $A$. For  each $a\in A$,   $\mathsf{E}(a)$  turns out to coincide with
the Euler characteristic of  the open set of
maximal ideals $\mathfrak m$ of $A$ such that $a/\mathfrak m$
is nonzero.
\end{abstract}

\maketitle

  \hfill{\it to J\'an Jakub\'ik, on his 90th birthday}

\section{Introduction}
For all unexplained notions concerning
MV-algebras we refer to \cite{cigdotmun} and \cite{mun11}.
We refer to \cite{mau} for background on algebraic topology.
A {\it rational polyhedron} in $\I^d$
 is a finite union
of simplexes with rational vertices in $\I^d.$
As shown in \cite[\S 4]{marspa} and  \cite[\S 3.4]{mun11}, 
rational polyhedra are dually equivalent to
finitely presented MV-algebras.
In the light of
 \cite[6.3]{mun11}, any finitely presented MV-algebra $A$ 
 can be identified with the MV-algebra  $\McN(P)$ of all McNaughton
functions over some rational polyhedron $P\subseteq \I^d$, for
some  $d=1,2,\ldots .$
In \cite{cabmun} and  \cite[6.3]{mun11}
 it is proved that $A$ is finitely presented iff it has a basis
$B$. An element $x\in A$  is said to be {\it basic} if it belongs
to some basis of $A$. 
We let ${\mu}(A)$ denote the maximal spectral space of the
MV-algebra $A$. By   \cite[4.16]{mun11},
for every maximal ideal  $\mathfrak m \in {\mu}(A)$
the quotient MV-algebra  $A/\mathfrak m$  is uniquely
isomorphic to a subalgebra  $J$ of the standard MV-algebra
$\I.$  Identifying $A/\mathfrak m$ and $J$, 
for any $a\in A$, 
the element  $a/\mathfrak m$ becomes a
 real number. We write  
$
\,\,\,\,\support(a)=\{\mathfrak m\in{\mu}(A)\mid a/\mathfrak m >0\}.\,\,\,
$
When  $A=\McN(P)$ is finitely presented,
 $\support(a)$ is
homeomorphic to the complementary set in $P$ of
a rational polyhedron   in $\I^d$
 (see \cite[\S 4.5, 6.2]{mun11}). 
By definition, the Euler characteristic
 $\chi({\support(a)})$ is 
  the alternating sum of the Betti numbers of 
  $\support(a)$,
as given by singular homology theory. 
This is homotopy invariant. For any piecewise linear 
continuous function  $l \colon \cube\to \I$
the set $\{x\in \cube\mid l(x)>0\}$ is homotopy equivalent
to $\{x\in \cube\mid l(x)\geq \epsilon\}$ for all small
enough $\epsilon >0$: an
 exercise in 
(piecewise linear) Morse theory,
\cite{morse, for} shows that the latter is a deformation
retract of the former.  
 As a consequence, 
\begin{equation}
\label{equation:india}
\chi(\{x\in \cube\mid l(x)>0\})=
\chi(\{x\in \cube\mid l(x)\geq \epsilon\}),\,\forall  \epsilon>0  \mbox{ small enough. }
\end{equation} 
As an  MV-algebraic
variant of the main result of  \cite{ped},
 in   this paper we prove:

\begin{theorem}
\label{theorem:pedrini}
For any finitely presented MV-algebra $A$ let  the map
   $\mathsf{E}\colon A\to \mathbb Z$  be   given  by
$
\mathsf{E}(a)=\chi({\support(a)}),\,\,\,\mbox{for all}\,\,a\in A.
$
Then $\mathsf E$
 has the following properties:
\begin{itemize}
\item[(i)]  $\mathsf{E}(0)=0$.

\smallskip
\item[(ii)]   {\rm  (Normalization)} $\mathsf{E}(b)=1$ for each basic element $b$ of $A$.

\smallskip
\item[(iii)]     {\rm (Idempotency)} For all $p,q\in A$,
$\mathsf{E}(p \oplus q)=\mathsf{E}(p\vee q)$.

\smallskip
\item[(iv)]   {\rm  (Additivity)}
 $\mathsf{E}$ is a {\rm valuation}:
 for all $p,q\in A$, $\mathsf{E}(p \vee q)=\mathsf{E}(p)+\mathsf{E}(q)-\mathsf{E}(p \wedge q).$
\end{itemize}
Conversely, 
 properties (i)-(iv) uniquely characterize $\mathsf{E}$ among all
real-valued functions
 defined on $A$.
\end{theorem}
The main reason of interest for this result, and of
distinction from \cite{ped}, 
is the combination  of  MV-algebraic base theory \cite[\S\S 5,6]{mun11} with
piecewise linear Morse theory \cite{morse, for}
to  the construction of a Turing computable
Euler valuation on every finitely presented
MV-algebra  $A$, and its characterization as
the Euler characteristic $\chi$ of the support of
all $a\in A.$
 The proofs in this paper use
sophisticated
MV-algebraic techniques that are not available
in the context of Riesz spaces of \cite{ped}.

\section{MV-algebraic lemmas}
\begin{lemma}{\rm(\cite[6.3]{mun11} and \cite[6.4]{mun11})}
\label{lemma:prerequisite} $A$ is a finitely presented
MV-algebra iff  $A$ has a basis off $A$ has the form  $\McN(P)$ for some
rational polyhedron $P\subseteq\I^d$ in some euclidean space 
${\mathbb R}^d$. 
 Let $\mathcal B$ be a basis of $A$.
Then $P$ can be chosen so that
 there is a regular triangulation $\Delta$ of $P$ and
an isomorphism of $A$ onto $\McN(P)$ sending the elements
of $\mathcal B$ one-one onto the elements of the Schauder basis
$\mathcal H_\Delta.$
\end{lemma}

\begin{lemma}   
\label{basis-reduction}
Let $\mathcal B=\mathcal H_{\Delta}$ be a  
Schauder basis of  an
 MV-algebra $A=\McN(P)$, for some regular
 triangulation $\Delta$ of $P$. 
Let  $b_0,b_1,\ldots,b_u$  be distinct elements of   $\mathcal B$,
with their respective vertices
 $v_0,v_1,\ldots,v_u$.
Let us define
$b_0^0=b_0,\,\mathcal B^0=\mathcal B, \, \Delta(0)=\Delta,$
and inductively, 
$c_k=b_k \wedge b_0^{k-1},\,\,
b_0^k=b_0^{k-1}\ominus c_k,\,\, a_k=b_k\ominus c_k,$
and
$$
\mathcal B^k = (\mathcal B^{k-1} \setminus \{b_0^{k-1},b_k\})
\cup\{a_k, b_0^k\}\cup\{c_k, \text{ if }c_k\neq0\},
\,\,\, k=1,\dots,u.
$$
\begin{itemize}
\item[(i)]
Then  
$b_0\wedge \bigoplus_{i=1}^u b_i$ coincides with 
$c_1\oplus\cdots\oplus c_u$, where any two nonzero
$c_i,c_j$ are distinct elements of  
$ \mathcal B^u$.
\item[(ii)]  For each  $k=1,\ldots,u,\,\,\,
\mathcal B^k=\mathcal H_{\Delta(k)}$, where
the regular triangulation
$\Delta(k)$ is obtained by blowing up
$\Delta(k-1)$ at the Farey mediant $w_{k}$ of
the segment  $\conv(v_0,v_{k})$, 
provided 
  $\,\,\conv(v_0,v_k)\in\Delta(k-1).$  Otherwise,   $\Delta(k)=\Delta(k-1).$
\end{itemize}

\begin{proof}
 The proof   proceeds by induction
  on $k=1,\ldots,u$. 
   If $k=1$, by definition, $b_0 \wedge b_1 = b_0^0 \wedge b_1 = c_1$. 
   If $c_1=0$, then upon letting
   $\mathcal B^1=\mathcal B$, (i)-(ii) are settled. 
If  $c_1\not=0$,  then
$\mathcal B^1=\mathcal B\setminus \{b_0,b_1\}
   \cup\{c_1,b_0\ominus c_1,b_1\ominus c_1\}$
   is the Schauder basis  $\mathcal H_{\Delta(1)}$ of the regular triangulation
   $\Delta(1)$ of $P$
   obtained from $\Delta$ via a Farey blow-up at the Farey mediant $w_1$ of the $1$-simplex $\conv(v_0,v_1)$ of $\Delta$.
  Trivially, (i) is satisfied.  
By  \cite[9.2.1]{cigdotmun},
  $\mathcal B^1=\mathcal H_{\Delta(1)}$ satisfies 
  (ii).

\medskip
Inductively, assume the
claim holds  for all $l<k$. 
In particular, $b_0\wedge \bigoplus_1^{k-1} b_i=\bigoplus_1^{k-1} c_i$,
with
    $c_1,\dots,c_{k-1}\in \mathcal B ^{k-1}\cup\{0\}$.
    Further,  $\mathcal B^{k-1}$  is a basis of $A$.
    
    \medskip
    
We will prove (i) arguing pointwise for each
$p\in P$:

\medskip
\noindent
    {\it Case 1:}\  $b_0\leq \bigoplus_{1}^{k-1}b_i\leq\bigoplus_{1}^{k}b_i$. 
    
     If there is $ j\in\{1,\dots,k-1\}$ such that $b_0^{j-1}\leq b_j$ at $p$,
           then $c_j=b_0^{j-1}$ and $b_0^{i-1}=c_i=0$, for all $i>j$. In particular $c_{k}=0$. Therefore,   $b_0\wedge \bigoplus_1^k b_i=b_0=b_0\wedge \bigoplus_1^{k-1} b_i=\bigoplus_{1}^{k-1}c_i=\bigoplus_{1}^{k}c_i$ at $p$, whence
the  identity (i) trivially holds
at $p$.
     
If there is no $ j\in\{1,\dots,k-1\}$ such that $b_0^{j-1}\leq b_j$ at $p$
then
  $b_i< b_0^{i-1}\leq b_0$ for all $i=1,\dots,k-1$. As a consequence, $c_i=b_i$ and $b_0^i=b_0-\bigoplus_{1}^{i}b_j$, for all $i=1,\dots,k-1$.  It follows that 
     $b_0^{k-1}=b_0-\bigoplus_{1}^{k-1}b_i=0$,
     and again $c_{k}=0$. 
Then
    $b_0\wedge \bigoplus_1^k b_i=b_0=b_0\wedge \bigoplus_1^{k-1} b_i=\bigoplus_{1}^{k-1}c_i=\bigoplus_{1}^{k}c_i$,
    and (i) is satisfied at $p$.

\medskip
\noindent
{\it Case 2:}\  $\bigoplus_{1}^{k-1}b_i\leq\bigoplus_{1}^{k}b_i<b_0$.

 By way of contradiction,  suppose
  there is a (smallest)  index $j$  
 such that $b_0^{j-1}< b_j$.
Then  $j>1$,  
  $\,\,\,b_i\leq b_0^{i-1}$,  and $\,\,\,c_i=b_i$, for all $i<j$. As a consequence, $b_0^j=b_0-\bigoplus_{1}^{j-1}b_i$, whence $b_0^j- b_j=b_0-\bigoplus_{1}^{j}b_i>0$ and $b_0^{j-1}> b_j$,
  which is a contradiction. 
   We have just shown that
    $b_i\leq b_0^{i-1}$ for all $i=1,\dots,k$. In particular $c_{k}=b_{k}$, whence 
$b_0\wedge \bigoplus_1^k b_i=\bigoplus_1^k b_i= \bigoplus_1^{k-1} b_i\oplus b_{k}=(b_0\wedge \bigoplus_1^{k-1} b_i)\oplus b_{k}=\bigoplus_1^{k-1} c_i\oplus c_{k}=\bigoplus_1^{k} c_i,$
thus   showing that
identity (i) holds  at $p$.

\medskip
\noindent
{\it Case 3:}\  $\bigoplus_{1}^{k-1}b_i<b_0\leq\bigoplus_{1}^{k}b_i$. 

As in the previous case, for all $ i=1,\dots,k-1$ we have $b_i\leq b_0^{i-1}$, whence $c_i=b_i$ and $b_0^i=b_0-\bigoplus_{1}^{i-1}b_j$. Moreover,  $b_0^{k-1}=b_0-\bigoplus_{1}^{k-1}b_i$ and $c_{k}=b_{k}\wedge(b_0-\bigoplus_{1}^{k-1}b_i).$
 Therefore,  
 $b_0\wedge \bigoplus_1^k b_i=(b_0-\bigoplus_1^{k-1}b_i\oplus\bigoplus_1^{k-1}b_i)\wedge (b_k\oplus\bigoplus_1^{k-1} b_i)=((b_0-\bigoplus_1^{k-1}b_i)\wedge b_k)\oplus \bigoplus_1^{k-1} b_i=c_k\oplus \bigoplus_1^{k-1} c_i=\bigoplus_1^{k} c_i$,
 which shows 
 that (i) holds at $p$.

\medskip     
     We have shown that identity (i) holds on $P$.

\medskip
We now prove that $\mathcal B^k$ satisfies (ii).
By  induction hypothesis, $\mathcal B^{k-1}=\mathcal H_{\Delta(k-1)}$,
where $\Delta(k-1)$ is obtained from $\Delta(0)=\Delta$
 via Farey blow-ups.

If $c_k=0$, then $\mathcal B^k=\mathcal B^{k-1}=\mathcal H_{\Delta(k-1)}$.
We show that $\Delta(k-1)=\Delta(k)$.
By way of contradiction, let $\conv(v_0,v_k)$ be a simplex of $\Delta(k-1)$.
Since  both $b_0^{k-1}$ and $b_k$ are linear 
 on $\conv(v_0,v_k)$,  then $v_0$ is the unique point
of $\conv(v_0,v_k)$ where $b_k$ vanishes, and $v_k$ is the unique point
of $\conv(v_0,v_k)$ where $b_0$
vanishes. Therefore,
 $c_k>0$ on the (nonempty) set
$\conv(v_,v_k)\setminus\{v_0,v_k\}$, a contradiction. 
Having thus shown that  $\conv(v_0,v_k)\not\in
\Delta(k-1)$, we obtain $\Delta(k-1)=\Delta(k)$ and $\mathcal B^k=\mathcal H_{\Delta(k)}$.

If  $c_k\neq0$  then by \cite[9.2.1]{cigdotmun}  
$\mathcal B^k=\mathcal H_{\Delta(k)}$.

The proof is complete.
 \end{proof}
\end{lemma}

 \begin{remark}
In Lemma \ref{basis-reduction}  each $\oplus$ operation is only applied
to  Schauder hats belonging to the same  basis. Therefore, 
each $\oplus$ symbol in the statement and in the proof of the 
lemma can be replaced by the  addition $+$ symbol.
Similarly,  in  all expressions
$b_0^{k-1}\ominus c_k$ and   $b_k\ominus c_k$,
each $\ominus$ operation 
can be replaced by the subtraction $-$ operation,
because it is always the case that
$c_k\leq b_0^{k-1}$ and $c_k\leq b_k$.
\end{remark}

\section{The open support of a finitely presented MV-algebra}

 Let  $A$ be a finitely presented MV-algebra.
 By { Lemma \ref{lemma:prerequisite}} we can identify
 $A$ with 
  $\McN(P)$
 for some rational polyhedron $P\subseteq\I^d.$
 For each element $a\in A$ the rational polyhedron
 $\oneset(a)
 \subseteq \I^d$ is defined by
$$
\mbox{oneset}(a)=\{x\in P\mid a(x)=1\}=\{\mathfrak m\in {\mu}(A)\mid a/\mathfrak m=1\}.
$$
For   each $\lambda\in \I,$ we also let
$$
 \mbox{uplevelset}_\lambda(a)=\{x\in P\mid a(x)\geq \lambda\}=
 \{\mathfrak m\in {\mu}(A)\mid a/\mathfrak m\geq
\lambda\}.
$$
Following \cite{mun11}, we let   $2\centerdot a=a\oplus a$
and inductively,  $(n+1)\centerdot a = a\oplus n\centerdot a.$

\begin{lemma}
\label{from-india}
With the above notation we have
\begin{equation}
\label{equation:global}
 \chi({\mu}(A))=\chi(\oneset(1))=\chi(\support(1))=\chi(P). 
\end{equation}
Further, for any $a\in A,$ the open set 
$\support(a)$
 is homotopy equivalent to the rational polyhedron
  $\uplevelset_{1/n}(a)$,   for all 
 large integers 
 $n$; as a matter of fact,  
  the latter   is a deformation retract of the former. 
Thus, in particular, 
$$
\chi({\support}(a)) = 
 \lim_{n\to\infty}\chi({\rm oneset}(n\centerdot a)).
$$
\end{lemma}

\begin{proof}  The verification of (\ref{equation:global}) is trivial.
To prove the remaining statements, for all $0<m\leq n$ we have  
$\oneset(m\centerdot a)\subseteq \oneset(n\centerdot a)$.
For all large $n$,   
${\rm oneset}((n+1)\centerdot a)$  collapses into
${\rm oneset}(n\centerdot a)$.
Thus all these rational polyhedra
are   homotopic, and 
$
\lim_{n\to\infty}\chi({\rm oneset}(n\centerdot a))
=\lim_{n\to\infty}\chi({\rm uplevelset}_{1/n}(a))
$
exists.
(Intuitively, a retraction is given by the map
sending each point at the boundary of $\oneset((n+1) \centerdot b)$
to the point at the boundary of $ \oneset(n \centerdot b)$  given
by the line originating from the vertex of the simplex
of the free face of this point.)

The proof   now follows from (\ref{equation:india}). 
\end{proof}

\begin{lemma}
\label{lemma:supplement} For each element
$a$ of a finitely presented MV-algebra $A=\McN(P)$,
the integer
 $\chi({\support}(a))$
 coincides with the Euler characteristic of the
  supplement in $P= {\mu}(A)$ of the set  
$$a^{-1}(0)=\{x\in P\mid a(x)=0\}
=\{\mathfrak m\in {\mu}(A)\mid a/\mathfrak m=0\}=P\setminus \support(a).$$
\end{lemma}

\begin{proof}
From the first part of the proof of
 \cite[5.3.9]{mau}.
 \end{proof}

%

\section{Proof of Theorem: uniqueness}
{ Lemma \ref{lemma:prerequisite}} 
yields  an integer $d>0$ together with a rational 
polyhedron $P\subseteq \I^d$  such that
$A$ can be identified with 
  $\McN(P)$ without loss of generality. 
Let $\mathsf{E}_1$  and $\mathsf{E}_2$ be real-valued functions
on $A$ satisfying (i)-(iv). 
{ By (i),} they agree at $0$.
Let $a$ be a nonzero element of $A$.
By \cite[\S 6.2,6.3]{mun11}, there is a regular triangulation  $\Delta$  of
 $P$ such that $a$  (is linear over every simplex of $\Delta$,
and)   can be written as  
 $$
 a=\bigoplus_{i=0}^u  m_i\centerdot b_i
 =\sum_{i=0}^u  m_i b_i,
 $$
for  distinct Schauder hats  $b_0, b_1,\ldots,b_u$
of $\Delta$  and integers
 $m_0, m_1,\ldots,m_u> 0.$  
{ Since $\mathsf{E}_2$ satisfies the idempotency condition  (iii), }
 $$
 \mathsf{E}_2(a)=\mathsf{E}_2(\sum_{i=0}^u  m_i b_i)=
\mathsf{E}_2( \sum_{i=0}^u  b_i)=
\mathsf{E}_2(b_0 + \sum_{i=1}^u  b_i)=
\mathsf{E}_2(b_0 \vee \sum_{i=1}^u  b_i).
 $$
 { Since $\mathsf{E}_2$ satisfies the additivity property (iv)}, 
 $$
 \mathsf{E}_2(b_0\vee \sum_{i=1}^u  b_i)=
 \mathsf{E}_2(b_0)+\mathsf{E}_2( \sum_{i=1}^u  b_i) - \mathsf{E}_2(b_0\wedge  \sum_{i=1}^u  b_i).
 $$
 By { Lemma \ref{basis-reduction}(i)}
  there is a Schauder basis  $\mathcal B'$ together
 with   elements  $b'_1,\ldots,b'_u\in \mathcal B'\cup\{0\}$  such that
 $b_0\wedge  \sum_{i=1}^u  b_i=\sum_{j=1}^u  b'_j$. Consequently,
$$
\mathsf{E}_2(a)= \mathsf{E}_2(b_0)+\mathsf{E}_2( \sum_{i=1}^u  b_i) - \mathsf{E}_2(\sum_{j=1}^u  b'_j).
$$
Arguing by induction on $u$ we conclude that  the
value  $\mathsf{E}_2(a)$ is computed by a linear polynomial function  $\rho$
 (uniquely determined by $\Delta$)
of the values of  $\mathsf{E}_2$ on finitely many basic elements of $A$.
The same holds for  $\mathsf{E}_1$, with the same  $\Delta$ and $\rho$.
   { By the normalization condition (ii),}  
$\mathsf{E}_1$ and $\mathsf{E}_2$ agree on basic elements,
whence  
 they agree on $a$.
 Therefore 
 $\mathsf{E}_1=\mathsf{E}_2$.

\section{End of proof of Theorem: $\mathsf{E}$ has properties 
 (i)-(iv)}
   %
%
%
 %
 %
 %
{  (i)}
 Trivially, $\chi(\support(0))=\chi(\emptyset)=0$.
 
\smallskip
{   (ii)}
  Since  $A$ is finitely presented, by \cite[6.3]{mun11}
$A$ has a basis $\mathcal B$.
Writing as above
$A=\McN(P)$, each basic element 
of $\McN(P)$  becomes
 a Schauder hat  $b$
in  the Schauder basis  $\mathcal B$. 
For  all large $n$  the rational polyhedron
$ \oneset(n \centerdot b)$
 is contractible to the vertex of the Schauder hat $b$.
 Therefore, 
 $\lim_{n\to\infty}\chi({\rm oneset}(n\centerdot b))=1$,
and
 { by Lemma \ref{from-india}}, 
  $\chi(\support(b))=1=\mathsf{E}(b)$, as desired.

\smallskip
{  (iii)}
Trivially,   $\support(p\vee q)
=\support(p\oplus q)$.

 \smallskip
{   (iv)}
In view of {  Lemma \ref{from-india}},
let the integer $n$ be so large that 
$$\mathsf{E}(p)=\chi(\oneset(n\centerdot p)),\,\,
 \mathsf{E}(q)=\chi(\oneset(n\centerdot q)), \,\,\mathsf{E}(p\vee q)=\chi(\oneset(n\centerdot (p\vee q))),$$
 and
$\,\mathsf{E}(p\wedge q)=\chi(\oneset(n\centerdot (p\wedge q))).$ 
Then we can write
\begin{eqnarray*}
\mathsf{E}(p)+\mathsf{E}(q)&=&
\chi(\oneset(n\centerdot p))+\chi(\oneset(n\centerdot q))\\
{}&=&\chi(\oneset(n\centerdot p)\cup \oneset(n\centerdot q))+
\chi(\oneset(n\centerdot p)\cap \oneset(n\centerdot q))\\
{}&=&
\chi(\oneset(n\centerdot p \vee n\centerdot q))+
\chi(\oneset(n\centerdot p \wedge n\centerdot q))\\
{}&=&
\chi(\oneset(n\centerdot(p\vee q)))+\chi(\oneset(n\centerdot(p\wedge q)))\\
{}&=&
\mathsf{E}(p\vee q)+\mathsf{E}(p\wedge q).
\end{eqnarray*}
 The proof of Theorem \ref{theorem:pedrini}  is now
 complete.
 \hfill $\Box$

\smallskip

By \cite[\S 3.4]{mun11},  every 
 finitely presented MV-algebra $A$
is  the Lindenbaum algebra
  of some
  formula  $\theta(X_1,\ldots,X_n)$ in
  \L ukasiewicz logic \L$_\infty$.
   The map $\mathsf{E}$ of Theorem \ref{theorem:pedrini}   determines
the integer-valued
 map $\mathsf{E}'$ from all formulas $\phi(X_1,\ldots,X_n)$,
by the stipulation
$\mathsf{E}'(\phi)=\mathsf{E}(\phi/\theta)$, where  $\phi/\theta$ is the equivalence
class of $\phi$ modulo $\theta.$  
  
  \begin{corollary}
  \label{corollary:turing}
 $\mathsf{E}'$ is Turing computable.
\end{corollary}

\begin{proof}
With reference to 
the { proofs of Lemmas \ref{basis-reduction}
and \ref{from-india}}  (or alternatively, Lemma \ref{lemma:supplement}),  
let $a\in A$ and let $\mathcal B$ a basis in $A$ such that 
 $a=\bigoplus_{0}^u  m_i\centerdot b_i
 =\sum_{0}^u  m_i b_i$,
for  distinct Schauder hats  $b_0, b_1,\ldots,b_u \in \mathcal B$
and integers $m_0, m_1,\ldots,m_u> 0$. Let $1 / d_i$ be the maximum 
value of $b_i$, for $i=0,\dots,u$. As in \cite[Definition 5.7]{mun11}, each
 $d_i$ is a  nonzero integer.
  Since the McNaughton function $a$ is linear on every simplex of $\mathcal B$, 
  each  nonzero local minimum
and each nonzero local maximum of $a$ is of
the form $m_i / d_i$, for some $i=0,\dots, u$. 
Let $n_i$ be the smallest integer such that $n_i \cdot m_i/d_i\geq1$,
and    $n=\max(n_0,\ldots,n_u)$.
For all integers $k\geq n$
    each nonzero local minimum and each nonzero local 
    maximum of $k\centerdot a$ is equal to $1$. As a consequence, $a$ has no local
    minima nor local maxima in $(0,1/n)$. 
     For each real number $0<\delta<1/(2n)$ the function $a$ can be slightly perturbed to obtain a smooth
   function $a_\delta$ with no local minima nor local maxima in $(\delta,1/n - \delta)$ and such that $\uplevelset_\lambda(a_\delta)$ is homotopy equivalent to $\uplevelset_\lambda(a)$, for any
   $\lambda\in(\delta,1/n - \delta)$. We can now apply Morse Theory (e.g \cite[Theorem 3.1]{mil}) to the functions 
    $\tilde{a}_\delta$, defined by $\tilde{a}_\delta(p)=-a_\delta(p)$ at each point $p$ where $a_\delta$ is defined, to obtain the homotopy equivalence of $\uplevelset_{\lambda_1}(a_\delta)$ and $\uplevelset_{\lambda_2}(a_\delta)$ for all $\lambda_1,\lambda_2\in(\delta,1/n - \delta)$.
     As a consequence,
    we get the homotopy equivalence of ${\rm uplevelset}_{1/k_1}(a)$ and ${\rm uplevelset}_{1/k_2}(a)$ for all $k_1,k_2>n$. This ensures that ${\rm oneset}(k_1\centerdot a)$ and ${\rm oneset}(k_2\centerdot a)$ are homotopy equivalent for all $k_1,k_2>n$. In conclusion,
  $\chi({\support}(a)) = \chi({\rm oneset}((n+1)\centerdot a))$. 
Perusal  of the proofs of
 {  Lemmas \ref{basis-reduction}
and \ref{from-india}/\ref{lemma:supplement}} in combination
with \cite[18.1]{mun11}, shows that given
(a formula for)  $a$, some Turing machine
will output
(formulas for)
the basis $\mathcal B$, along with the integer $n$,
and the integer
 $\mathsf{E}(a)= \chi({\rm oneset}((n+1)\centerdot a))$.
\end{proof}

%


\begin{thebibliography}{99}

\bibitem{morse}
A.A. Agrachev, D.Pallasche, S.Scholte,
On Morse theory for piecewise smooth functions,
Journal of dynamical and control systems, 3 (1997)  449--469.
 


\bibitem{cabmun}
L. Cabrer, D. Mundici,
Finitely presented lattice-ordered abelian groups with
order-unit,
Journal of Algebra,  343 (2011) 1--10.




\bibitem{cigdotmun}
{R.~L.~O. Cignoli, I.~M.~L. D'Ottaviano,  D. Mundici,}
{  Algebraic Foundations of
many-valued Reasoning},  Volume~7 of
         { Trends in Logic},
Kluwer, Dordrecht, and Springer, Berlin,  2000.


\bibitem{for}
R. Forman,
Morse theory for cell complexes,
Advances in Mathematics, 
134 (1998) 90--145.
 
 
 
 \bibitem{marspa}
  V. Marra,
L. Spada,
The dual adjunction between
MV-algebras and Tychonoff
spaces, Studia Logica, special issue
  in memoriam Leo Esakia,
  100 (2012) 253--278.
  
  
  
    \bibitem{mau}
C.R.F. Maunder, Algebraic Topology,
Cambridge University Press,  1980.
 
 
 
 \bibitem{mil}
 J. Milnor, Morse theory, 
 Princeton University Press, Princeton, New Jersey, 1963.
 
 


\bibitem{mun11}
      { D. Mundici,}
{Advanced \luk\ Calculus and MV-algebras},  Volume~35 of
         { Trends in Logic},
   Springer, Berlin, 2011.


 
 
         \bibitem{ped}
A. Pedrini,  The Euler characteristic of a polyhedron
as a valuation on its coordinate vector lattice, submitted to publication. 
Electronic version: 
arXiv:1209.3248v1,  14 Sep 2012.
 

 

\end{thebibliography}
 \end{document}